\documentclass[a4paper, 12pt]{amsart}
  \usepackage{amsmath,amsthm,amsfonts,amssymb}
  \theoremstyle{plain}
  \newtheorem{thm}{Theorem}[section]
  \newtheorem{lem}[thm]{Lemma}
   \newtheorem{cor}[thm]{Corollary}
  \newtheorem{que}[thm]{Question}
  \newtheorem{pro}[thm]{Proposition}
  \newtheorem{con}[thm]{Conjecture}
  \theoremstyle{definition}

\newcommand{\CC}{{\mathbb{C}}}

\newcommand{\Irr}{{\operatorname{Irr}}}

\newcommand{\cd}{{\operatorname{cd}}}

\let\leq=\leqslant
\let\geq=\geqslant

\title[Average element order and number of conjugacy classes]{The average element order\\ and the number of conjugacy classes\\ of finite groups}
 \author{E. I. Khukhro}
\address{Charlotte Scott Research Centre for Algebra, University of Lincoln, U.K., and \newline \indent  Sobolev Institute of Mathematics, Novosibirsk, 630090, Russia}
\email{khukhro@yahoo.co.uk}

 \author{A. Moret\'o}
 \address{Departament  de Matem\`atiques, Universitat de Val\`encia, 46100, Burjassot, Val\`encia Spain}
\email{Alexander.Moreto@uv.es}

 \author{M. Zarrin}
 \address{Department of Mathematics, University of Kurdistan, P.O. Box: 416,  Sanandaj, Iran}
\email{M.Zarrin@uok.ac.ir}

\date{}

\keywords{$p$-group, nilpotent group, number of conjugacy classes, element orders}
\subjclass[2010]{Primary 20D15, Secondary 20C15, 20E45}

\begin{document}

\begin{abstract}
Let $o(G)$ be the average order of the elements of $G$, where $G$ is a finite group. We show that there is no polynomial lower bound for $o(G)$ in terms of $o(N)$, where $N\trianglelefteq G$, even when $G$ is a prime-power order group and $N$ is abelian. This gives a negative answer to a question of A.~Jaikin-Zapirain.
\end{abstract}

\maketitle

\section{Introduction}

Let $G$ be a finite group, and let $o(G)$ be the average order of the elements of $G$, that is,
$$
o(G)=\frac{\sum_{g\in G}|g|}{|G|},
$$
where $|g|$ denotes the order of $g$. A.~Jaikin-Zapirain mentioned in \cite[p.~1134]{jai} that it would be very interesting to understand the relation between $o(G)$ and $o(N)$, where $N$ is a normal subgroup of $G$. More specifically, he posed the following question.

\begin{que}
\label{jz}
Let $G$ be a finite ($p$-) group and $N$ a normal (abelian) subgroup of $G$. Is it true that $o(G)\geq o(N)^{1/2}$?
\end{que}

Our main result provides a strong  negative answer to this question, which shows that one cannot obtain this type of inequality for any real number $c$ as the exponent on the right,  not just $1/2$.

\begin{thm}\label{t-a}
Let $c>0$ be a real number and let $p\geq 3/c$ be a prime. Then there exists a finite $p$-group $G$ with a normal abelian subgroup $N$ such that $o(G)<o(N)^c$.
\end{thm}

By the condition $p\geq 3/c$ in this theorem, putting $c=1/2$ we obtain a negative answer to Question~\ref{jz}  for all primes $p\geq 7$. But in fact a more careful consideration of the parameters involved gives a negative answer for all primes $p\geq 5$.

\begin{cor}\label{cor}
Question~\ref{jz} has a negative answer for any prime $p\geq 5$.
\end{cor}

After the current paper had been already largely written, we received from A.~Jaikin-Zapirain an unpublished earlier manuscript, where he constructed counterexamples to Question~\ref{jz} using Frobenius groups. This manuscript also stated the following questions of similar nature.

\begin{que}
\begin{itemize}
  \item[(a)]  Is there a constant $c_1 > 0$ such that for every finite group  $G$ we have
$o(G)\geq  c_1\cdot  mo(G)^{1/2}$, where  $mo(G)$ is the maximum order of elements of $G$?
  \item[(b)] Is there a constant $c_2 > 0$ such that for every finite group  $G$ and every normal subgroup $N$ of $G$ we have $o(G)\geq  c_2\cdot  o(N)^{1/2}$?
\end{itemize}
\end{que}

Our examples proving Theorem~\ref{t-a} also give negative answers  to these questions.

We have tried to understand the relevance of  Question~\ref{jz} to problems on the number of conjugacy classes of finite groups. Despite the fact that this question has a negative answer, our counterexamples do not disprove certain possible consequences, so we think it is worth mentioning these problems on the number of conjugacy classes in Section~2. In addition, the following question remains open.

\begin{que}
Fix a prime $p$.  Does there exist a number $c=c(p)>0$ depending on $p$ such that $o(G)\geq o(N)^c$ for any finite $p$-group $G$ and any (abelian) normal subgroup  $N$ of $G$?
\end{que}

In Section~3 we discuss prospective  counterexamples related to the Hughes conjecture and so-called anti-Hughes $p$-groups, which provided original motivation to our constructions. But the actual counterexamples are constructed from so-called secretive $p$-groups in Section~4. In Section~5  we discuss other questions related to the function  $o(G)$, the recently intensely studied function $\psi(G)=\sum_{g\in G}|g|$, and the number of conjugacy classes and irreducible character degrees.

\section{Implications of Question \ref{jz}}

As usual, given a finite group $G$ we write $k(G)$ to denote the number of conjugacy classes of $G$. Also, if $L\leq K$ are normal subgroups of $G$ with $K/L$ abelian, we say that $K/L$ is an abelian section of $G$.
Using Lemma 2.6 and Corollary 2.10 of \cite{jai}, it is easy to see that if Question \ref{jz} had an affirmative answer, then the following question would also have an affirmative answer.

\begin{que}\label{q21}
Let $p$ be a prime. Let $G$ be a finite group with an abelian section of exponent $p^e$. Is it true that $k(G)\geq\sqrt{p^e-p^{e-1}}$?
\end{que}

In particular, this implies that if $G$ is a $p$-solvable group  and $p$ divides $|G|$, then $k(G)\geq\sqrt{p-1}$. This is a first consequence of Question \ref{jz} that is, we think, nontrivial. It had been proved in \cite{hk00} that for $G$ solvable, $k(G)\geq2\sqrt{p-1}$ and this bound is best possible. This paper has generated a lot of research since 2000.  It was mentioned in \cite{hk00} that if this bound holds for $p$-solvable groups and the  Alperin--McKay conjecture holds, then this bound holds for arbitrary groups. In fact, it suffices to assume that McKay's conjecture holds:
$$
k(G)=|\Irr(G)|\geq|\Irr_{p'}(G)|=|\Irr_{p'}(N_G(P))|=|\Irr(N_G(P)/P')|=k(N_G(P)/P'),
$$
and the last group is $p$-solvable. Here,  $\Irr(G)$ stands for the set of irreducible characters of a group $G$, and $\Irr_{p'}(G)$ is the set of irreducible characters of $G$ of degree not divisible by $p$. The equality $|\Irr_{p'}(G)|=|\Irr_{p'}(N_G(P))|$ is McKay's conjecture, which is one of the main problems in representation theory.
Much progress has been made recently on McKay's conjecture. It was reduced to a problem on simple groups in \cite{imn} and the case $p=2$ was proved in \cite{ms}, but the conjecture is (as of now) open for odd $p$. See \cite{nav} for a detailed exposition of the results and techniques that have been used.

 It was not until 2016 that the bound $k(G)\geq2\sqrt{p-1}$ was proved for arbitrary groups  in \cite{mar}. It was shown subsequently in \cite{mm} that $|\Irr_{p'}(G)|\geq2\sqrt{p-1}$ for arbitrary finite groups, a result that is consistent with McKay's conjecture. In view of this, one could think that, perhaps,
$|\Irr_{p'}(G)|\geq\sqrt{p^e-p^{e-1}}$. However the dihedral $2$-groups, for instance, are already a counterexample for $p=2$. This means that Question \ref{q21} does not follow from the $p$-solvable case and McKay's conjecture.

See the introduction of \cite{ms} for a summary of the developments in this area. In that paper, A. Mar\'oti and I. Simion  proved that there exists a constant $c>0$ such that $k(G)\geq cp$ for any finite group $G$ of order divisible by $p^2$.
However, it is not true that if we assume that $p^n$ divides $|G|$, then $k(G)$ grows (exponentially) with $n$, even in prime-power order groups: as mentioned in Remark (v) of \cite{hk00}, L. Pyber pointed out that L. Kovacs and C. R. Leedham-Green \cite{kl} had constructed for every odd prime $p$ a group $G$ of order $p^p$ and of exponent $p$ with $k(G)=\frac{1}{2}(p^3-p^2+p+1)$.
A lower bound  for $o(G)$ in terms of $o(N)$, for any $N$ normal abelian in $G$,  would imply that $k(G)$ does grow with the exponent of an abelian section of $p$-power order. There is one recent result of this type, although of a  different nature. Let $B$ be a Brauer $p$-block of a finite group $G$ with defect group $D$. It was proved in Corollary 4 of \cite{oto} that if $\exp(Z(D))=p^m$, then $k(G)\geq k(B)\geq (p^m+p-2)/(p-1)$. (Here, $\exp(H)$ denotes the exponent of a group $H$.) In short, all of this seems to indicate that $k(G)$ grows with exponents. For instance, we propose the following question.

\begin{que}\label{q22}
Let $G$ be a $p$-group. Is it true that $k(G)\geq\exp(G)^{1/2}$?
\end{que}

This would be an immediate consequence of Corollary 2.10 of \cite{jai} and an affirmative answer to Question~\ref{zarrin} below, but we know that the answer to the latter  question is negative. If Question~\ref{q22} has an affirmative answer, it  will show that examples of $p$-groups with few classes like those of Kovacs and Leedham-Green necessarily have ``low" exponent.

Our counterexamples in this paper do not provide counterexamples to the questions posed in this section. If these bounds turned out to be false, it would still be interesting to see if polynomial bounds do exist. Our counterexamples show that if these bounds exist, they are not immediate consequences of a general result relating $o(G)$ and $o(N)$, where $N$ is normal abelian subgroup of~$G$.

\section{Anti-Hughes groups}\label{s-a-h}

Trying to prove that Question \ref{jz} has an affirmative answer for $p$-groups, we considered the following question. Recall that $\exp(H)$ denotes the exponent of a group $H$.

\begin{que}
\label{zarrin}
Let $G$ be a finite $p$-group. Is it true that $o(G)\geq\exp(G)^{1/2}$?
\end{que}

 Since $o(H)\leq\exp(G)$ for any subgroup $H$ of $G$, it is clear that an affirmative answer to Question \ref{zarrin} for a finite $p$-group $G$ implies an affirmative answer to Question \ref{jz} for the same group $G$.

Let us recall the Hughes conjecture. Given a finite group $G$ and a prime $p$, we define $H_p(G)=\langle x\in G\mid x^p\ne 1 \rangle$. D.~R.~Hughes \cite{hug} conjectured that if $1\neq H_p(G)\neq G$, then $|G:H_p(G)|=p$. Hughes and J.~G.~Thompson \cite{ht} proved that the Hughes conjecture holds for all groups that are not $p$-groups. Since then, a lot of effort has been devoted to the study of the Hughes conjecture for $p$-groups. There are a number of results in both directions. For instance, the conjecture is known to be true for $p=2$ and $p=3$. On the other hand, there are counterexamples for $p=5,7,11,13,17,19$ (and possibly a few more primes). A counterexample to the Hughes conjecture is called an anti-Hughes group. It is expected that counterexamples exist for every prime $p\geq5$,  but it is also known that the Hughes conjecture holds ``for almost all groups" in a certain precise sense \cite{khu88}; in particular, the exponent of an anti-Hughes group is bounded in terms of $p$. In all known counterexamples $|G:H_p(G)|=p^2$ and $\exp(G)=p^2$, but it is expected that there exist counterexamples with $|G:H_p(G)|>p^2$ and also counterexamples with $\exp(G)>p^2$. We refer the reader to \cite{khu93} and \cite{khu99} for  more detailed expositions on the Hughes conjecture and to \cite{hv} for the most recent paper where more counterexamples were built with the aid of computers. (Only the first counterexampe by G. E. Wall \cite{wal67} was constructed manually for $p=5$.)

Notice that by construction all the elements in $G\setminus H_p(G)$ have order $p$. Therefore, if $G$ is an anti-Hughes group, then the proportion of elements of order $p$ is at least $(p^2-1)/p^2$. It was natural to check if we can have a counterexample to Questions \ref{zarrin} and \ref{jz} among anti-Hughes groups. Notice however that these questions hold trivially for groups of exponent $p^2$. Therefore, there are no counterexamples to our questions among the known anti-Hughes groups. However, we have the following.

\begin{pro}
Assume that there exists an anti-Hughes group $G$ of exponent $p^3$. Then $G$ is a counterexample to Question \ref{zarrin}
\end{pro}

\begin{proof}
Write $|G|=p^n$.
As mentioned, all elements in $G\setminus H_p(G)$ have order~$p$. The elements in $H_p(G)$ have order at most $p^3$. Therefore,
$$
o(G)\leq \frac{(p^n -p^{n-2})p+ p^{n-2}p^3}{p^n}<2p<(p^3)^{1/2}=\exp(G)^{1/2},
$$
where in the last inequality we have used the fact that in an anti-Hughes group $p\geq 5$.
\end{proof}

A more detailed analysis of the orders of the  elements in $H_p(G)$ would be required to see if such a group is also a counterexample to Question  \ref{jz}. Fortunately there is a related, but simpler, construction that will allow us to find counterexamples. These are the so-called secretive $p$-groups (see Remark~4.8 of \cite{khu99}). Thanks to them, we can construct counterexamples for almost all primes.

\section{Secretive 
$p$-groups and counterexamples}

The key to our construction will be the following groups.

\begin{lem}[Wall's secretive $p$-group {\cite{wal75}}]
\label{wall}
For every prime $p$, there exists a finite $p$-group $P$ such that 
\begin{itemize}
  \item[\rm (a)] the Frattini subgroup $\Phi (P)$ has exponent $p$,
  \item[\rm (b)] $|P/\Phi (P)|=p^p$, and 
  \item[\rm (c)] all the elements in $P\setminus \Phi(P)$ have order $p^2$ and their $p$-th powers are nontrivial elements of a central cyclic subgroup $\langle z\rangle=P^p$ of order $p$.
\end{itemize}
\end{lem}

These groups are special cases of a somewhat more general construction of G.~E.~Wall in \cite{wal75}, where the exponent of the group was $p^m$ for $m\geq 2$ and the rank of $P/\Phi (P)$ could be any integer between 2 and $p^{m-1}$. It is not explicitly stated in \cite{wal75} that  $\Phi (P)$ has exponent $p$
in the group $P$ in Lemma~\ref{wall}, only that $P$ has exponent $p^2$ and $P/[P,P]$ has exponent $p$; but it is straightforward from the construction in \cite{wal75} that $[P,P]$ has exponent $p$.

These groups are   so-called secretive $p$-groups, and their construction in \cite{wal75} answered a question of L.~Kovacs, J.~Neub\"user and B.~Neumann \cite{knn}. For $m=2$ the first author \cite{khu86} also constructed such groups with even larger rank of $P/\Phi(P)$, which gave a negative  answer  to a question of N.~Blackburn and A.~Espuelas \cite{be}. The construction in \cite{khu86} produced $p$-groups $P$ with $\Phi(P)=\Omega_1(P)$ of exponent $p$ and $|P^p|=p$ and with the rank of  $P/\Phi(P)$ being $1+k(p-1)$  under  the condition that the multilinear identities  in the associated Lie ring $L(B(\infty, p))$ of the free countably-generated (restricted) Burnside group of exponent $p$  are `really new' for every degree $1+r(p-1)$ for $r=1,2,\dots,k$.  `Really new' here means that they are not consequences of identities of smaller degrees. All multilinear identities  in $L(B(\infty, p))$  were described by M.~R.~Vaughan-Lee \cite{v-l}, and G.~E.~Wall \cite{wal86}  proved that only those of degrees  $1+r(p-1)$ can be really new. Wall \cite{wal86}    showed that the aforementioned condition also implies the existence of anti-Hughes groups with $|P:H_p(P)|=p^{1+k(p-1)}$.

At the moment such identities are known to be really new only for degrees $p$ (the $(p-1)$-Engel identity) for all $p$, and $2p-1$ (Wall's identity) for a few primes $p=5,7,11, \dots, 19$ (possibly a few more), which was established with the aid of computers (only for $p=5$ Wall's calculations were done manually in \cite{wal67}). Notably, the construction in \cite{khu86} of secretive $p$-groups of rank $1+k(p-1)$ succeeds under the same conditions on the  identities  of $L(B(\infty, p))$ that ensure the existence of anti-Hughes groups, as in \cite{wal67}, \cite{wal74}, \cite{wal86}. 
However, for the counterexamples in the present paper it is  sufficient for us to use Wall's secretive $p$-groups from Lemma~\ref{wall}, which exist for every prime $p$. Secretive $p$-groups with bigger rank of $P/\Phi(P)$ may still prove to be useful for refuting weaker conjectures in the future. The same applies to the prospective anti-Hughes groups of exponent greater than $p^2$ as explained in \S\,\ref{s-a-h}, as well as anti-Hughes groups with bigger index of $H_p(P)$.
However, for these more difficult constructions the underlying problem of really new identities in $L(B(\infty, p))$ must be tackled first.

We now proceed to the construction of our counterexamples.

\begin{lem}
\label{construction}
Let $P$ be Wall's secretive $p$-group from Lemma~\ref{wall} and let $s>0$ be an integer. Then there exists a homocyclic group $U_s$ of exponent $p^s$
admitting an action of $P$ by automorphisms such that in the semidirect product $U_sP$ the following hold:
\begin{enumerate}
\item[\rm (a)]
the order of any element in $U_sP\setminus U_s\Phi(P)$ is $p^2$;
\item[\rm (b)]
the order of any element in $U_s\Phi(P)$ is at most $p^{s+1}$.
\end{enumerate}
\end{lem}

\begin{proof}
Recall that $P^p=\langle z\rangle$ is a central subgroup of $P$ of order $p$.
Consider a complex irreducible representation of $P$ in which $\langle z\rangle$ is represented faithfully, and let $V$ be the corresponding vector space over $\CC$. We use the usual centralizer notation for the natural semidirect product $VP$  regarding $V$ as an additive group. Since $z\in Z(P)$ and $V$ is an irreducible $\CC P$-module, we have $C_V(z)=0$ and therefore
\begin{equation}\label{e-split}
  v+vz+vz^2+\dots +vz^{p-1}=0\qquad \text{for any }v\in V.
\end{equation}

Pick an arbitrary non-zero vector $a\in V$ and let
$$
A={}_+\langle ag\mid g\in P\rangle
$$
be the additive subgroup of $V$ generated by the orbit of $a$ under the action of~$P$. As a finitely generated subgroup of the torsion-free additive group $V$, the group $A$ is a free abelian group of finite rank. We now put $U_s=A/p^sA$, which is a finite homocyclic $p$-group of exponent $p^s$ admitting the induced action of~$P$.

We now switch to multiplicative notation for the group operations in $U_s$ and the semidirect product $U_sP$.
Note that by \eqref{e-split} we have a similar property
\begin{equation*}
  u\cdot u^z\cdot u^{z^2}\cdots u^{z^{p-1}}=1\qquad \text{for any }u\in U_s.
\end{equation*}
Since $z^p=1$, the last equation is well known to be equivalent to
\begin{equation*}
(uz)^p= u\cdot u^z\cdot u^{z^2}\cdots u^{z^{p-1}}=1\qquad \text{for any }u\in U_s.
\end{equation*}
Of course, the same holds for any non-trivial power of $z$:
\begin{equation}\label{e-split3}
(uz^k)^p= 1\qquad \text{for any }u\in U_s \text{ and any }k=1,2,\dots,p-1.
\end{equation}

Any element  $g\in U_sP\setminus U_s\Phi(P)$ has the form  $g=wh$ for $w\in U_s$ and $h\in P\setminus \Phi(P)$. Then
$$
g^{p}=(wh)^p=vh^p
$$
for some $v\in U_s$, while $h^p=z^k\ne 1$ by the property of Wall's $p$-group $P$ in Lemma~\ref{wall}. Therefore by \eqref{e-split3} we have
$$
g^{p^2}=((wh)^{p})^p=(vz^k)^p=1,
$$
as required.

It is also obvious that the exponent of $U_s\Phi (P)$ is at most $p^{s+1}$, since  the exponent of $U_s$ is $p^s$ and  the exponent of  $\Phi(P)$ is $p$ by Lemma~\ref{wall}.
\end{proof}

We now complete the proof of Theorem~\ref{t-a}.

\begin{proof}[Proof of Theorem~\ref{t-a}]
Recall that $c$ is a positive real number, and $p$ is a prime such that $p>3/c$. We need to construct a finite $p$-group $G$ and an abelian normal subgroup $N$ such that $o(G)<o(N)^c$.

Set $s=p+1$ in Lemma \ref{construction}. We set $G=U_sP$ and $N=U_s$. Write $|G|=p^n$. Using the estimates for the element orders in (a) and (b) of Lemma \ref{construction}, we have
$$
o(G)\leq\frac{p^2(p^n-p^{n-p})+p^{p+2}p^{n-p}}{p^n}=p^2-p^{2-p}+p^2<2p^2\leq p^3
$$
On the other hand, for example, by   \cite[Lemma 2.6]{jai},
$$
o(N)\geq p^{p+1}-p^p\geq p^p.
$$
Therefore, since $3\leq pc$, we have
$$
o(G)\leq p^3\leq p^{pc}\leq o(N)^c,
$$
as desired.
\end{proof}

\begin{proof}[Proof of Corollary~\ref{cor}]
 Since $7>3/(1/2)$, Theorem~\ref{t-a} provides counterexamples to Question~\ref{jz} for all primes $p\geq 7$. In fact, the same example in the proof of Theorem~\ref{t-a} works also for $p=5$ and $c=1/2$. Indeed, for these values we have
$o(G)<2\cdot 5^2$ as above, which is less than $(5^{5+1}-5^5)^{1/2}\leq o(N)^{1/2}$.
\end{proof}

\section{Related questions}

The function $o(G)$ may have interest by itself. The analogous functions for character degrees and class sizes have been studied (see, for instance,  \cite{mn} or \cite{gr}). By analogy with Theorem A of \cite{mn} and Theorem 11 of \cite{gr}, we propose the following.

\begin{con}
\label{a5}
Let $G$ be a finite group. If $o(G)<o(A_5)$ then $G$ is solvable.
\end{con}

Notice that if we set $\psi(G)=\sum_{g\in G}|g|$, then $o(G)=\psi(G)/|G|$. The function $\psi$ was defined in \cite{aai}, where it was proved that if $G$ is a group of order $n$  then $\psi(G)<\psi(C_n)$. Many papers on this function have been published since then. For instance, in \cite{hlm} it was conjectured that if a group $G$ has order $n$ and  $\psi(G)>\frac{211}{1617}\psi(C_n)$ then $G$ is solvable. (Notice that $\psi(A_5)=211$ and $\psi(C_{60})=1617$.) This conjecture was proved in \cite{bk}. We can reformulate Conjecture \ref{a5} in terms of $\psi$ as follows: if $G$ is a group of order $n$ and $\psi(G)<\frac{211}{60}n$ then $G$ is solvable. Notice that this would be a dual result to \cite{bk}.

Although not directly related with $o(G)$, we take this opportunity to record the following consequence of the main result of \cite{jai}. It was asked in \cite[Question 1 and p. 246]{mor} whether $k(G)/|\cd(G)|\to\infty$ when $G$ is solvable and $|G|\to\infty$, or even $k(G)/d(|G|)\to\infty$ when $G$ is solvable and $|G|\to\infty$. (Here $\cd(G)$ stands for the number of irreducible character degrees of $G$ and $d(|G|)$ is the number of divisors of $|G|$.) These questions have an affirmative answer for nilpotent groups.

\begin{pro}
We have
\begin{enumerate}
\item
$\frac{k(G)}{d(|G|)}\to\infty$ when $G$ is a nilpotent and $|G|\to\infty$.
\item
In particular, $\frac{k(G)}{|\cd(G)|}\to\infty$ when $G$ is nilpotent and $|G|\to\infty$.
\end{enumerate}
\end{pro}

\begin{proof}
We consider first the case of prime power order groups. Write $|P|=p^m$, where $p$ is a fixed prime.  Note that $d(|P|)=m+1=\log_p|P|+1=\frac{1}{\log_2p}\log_2|P|+1$.  If we use the lower bound for $k(P)$ given in Theorem 1.1 of \cite{jai}, we obtain that $\frac{k(P)}{d(|P|)}\to\infty$  when $m\to\infty$.

Now, assume that $|P|=p^m$, where $m$ is fixed. In this case, $k(P)\geq|Z(P)|\geq p$, so
$$
\frac{k(P)}{d(|P|)}\geq\frac{p}{m+1}\to\infty\text{\,\,  when $p\to\infty$}.
$$

Finally, notice that if $G=P_1\times\cdots\times P_t$ is nilpotent with Sylow subgroups $P_1,\dots, P_t$ then
$$
\frac{k(G)}{d(|G|)}=\frac{k(P_1)}{d(|P_1|)}\cdot\cdots\cdot\frac{k(P_t)}{d(|P_t|)},
$$
using that $d$ is a multiplicative function. Now, it is a calculus exercise to check that the result follows from the previous paragraphs and the fact that $\frac{k(P)}{d(|P|)}\geq1$ for every $p$-group $P$.
\end{proof}

\section*{Acknowledgements}
\noindent The work of the first author was supported by Mathematical Center in Akademgorodok, the agreement with Ministry of Science and High Education of the Russian Federation no.~075-15-2019-1613. The research of the second author is  supported by Ministerio de Ciencia e Innovaci\'on PID-2019-103854GB-100 and FEDER funds.


\begin{thebibliography}{131}

\bibitem{aai}
H. Amiri, S. M. Jafarian Amiri, and I. M. Isaacs, Sums of element orders in finite groups, \textit{Comm. Algebra} \textbf{37} (2009), 2978--2980.


\bibitem{bk}
M. Baniasad Azad and B. Khosravi, A criterion for solvability of a finite group by the sum of element orders, \textit{J. Algebra} \textbf{516} (2018), 115--124.

\bibitem{be} N. Blackburn and A. Espuelas, The power structure of metabelian  $p$-groups, \textit{Proc. Amer. Math. Soc.} \textbf{92}, no.~4 (1984), 478--84.


\bibitem{gr} R. Guralnick and G. Robinson,
On the commuting probability in finite groups, \textit{J. Algebra} \textbf{300} (2006), 509--528; Addendum, \textit{J. Algebra} \textbf{319} (2008), 1822.

\bibitem{hv} G. Havas and M. Vaughan-Lee, On counterexamples to the Hughes conjecture, \textit{J. Algebra} \textbf{322} (2009), 791--801.

\bibitem{hlm} M. Herzog, P. Longobardi, and M. Maj, Two new criteria for solvability of finite groups,  \textit{J. Algebra}  \textbf{511} (2018), 215--226.


\bibitem{hk00} L. Hethelyi and B. K\"ulshammer, On the number of conjugacy classes of a finite solvable group, \textit{Bull. London Math. Soc.} \textbf{32} (2000), 668--672.

    \bibitem{hug} D. R. Hughes, A research problem in group theory, \textit{Bull. Amer. Math. Soc.} \textbf{63} (1957), 209.

    \bibitem{ht} D. R. Hughes and J. G. Thompson, The $H_p$-problem and the structure of $H_p$-groups, \textit{Pacific J. Math.} \textbf{9} (1959), 1097--1101.

    \bibitem{imn} I. M. Isaacs, G. Malle, G. Navarro, A reduction theorem for the McKay conjecture, \textit{Invent. Math.} \textbf{170} (2007), 33-101.

\bibitem{jai} A. Jaikin-Zapirain, On the number of conjugacy classes of finite nilpotent groups, \textit{Adv. Math.} \textbf{227} (2011), 1129--1143.

\bibitem{khu86}
E. I. Khukhro, Finite $p$-groups that are close to groups of prime period,  \textit{Algebra  Logika} \textbf{25} (1986), 227--240; English transl. in  \textit{Algebra Logic} \textbf{25}  (1987), 143--153.

\bibitem{khu88}
E. I. Khukhro, On the Hughes problem for finite $p$-groups, \textit{Algebra Logika} \textbf{26} (1987), 642--646; English transl. in \textit{Algebra Logic} \textbf{26}  (1988), 398--401.

\bibitem{khu93} E. I. Khukhro, \textit{Nilpotent groups and their automorphisms}, de Gruyter, Berlin, 1993.

\bibitem{khu99} E. I. Khukhro, Generalizations of the restricted Burnside problem for groups with automorphisms, in: \textit{Groups St. Andrews 1997 in Bath}, II, 474–491, London Math. Soc. Lecture Note Ser., vol. \textbf{261}, Cambridge Univ. Press, Cambridge, 1999.

\bibitem{kl} L. G. Kovacs, C. R. Leedham-Green, Some normally monomial $p$-groups of maximal class and large derived length, \textit{Quart. J. Math. Oxford Ser.} \textbf{37} (1986), 49--54.

\bibitem{knn}  L. G. Kovacs, J. Neubuser, and B. H. Neumann, On finite groups with
hidden primes, \textit{J. Austral. Math. Soc.} \textbf{12} (1971), 287--300.

\bibitem{mm} G. Malle and A. Mar\'oti, On the number of $p'$-degree characters in a group, \textit{Int. Math. Res. Not.} \textbf{20} (2016), 6118-6132.

\bibitem{ms} G. Malle and B. Sp\"ath, Characters of odd degree,\textit{ Ann. of Math.} \textbf{184} (2016), 869-908.

\bibitem{mar}
A. Mar\'oti, A lower bound for the number of conjugacy classes of a finite group,
\emph{Adv. Math.} \textbf{290} (2016), 1062--1078.

\bibitem{ms}
A. Mar\'oti and I. Simion, Bounding the number of classes of a finite group in terms of a prime,
\emph{J. Group  Theory} \textbf{23} (2020), 471--488.


\bibitem{mor} A. Moret\'o,  Complex group algebras of finite groups: Brauer's Problem 1, \textit{Adv. Math.} \textbf{208} (2007), 236--248.

\bibitem{mn}
A. Moret\'o and H. N. Nguyen, On the average character degree of finite groups, \textit{Bull. London Math. Soc.} \textbf{46} (2014), 454--462.

\bibitem{nav} G. Navarro,  \textit{Character theory and the McKay conjecture}, Cambridge University Press, Cambridge, 2018.

\bibitem{oto}
Y. Otokita, On Loewy lengths of centers of blocks, in: \textit{Proc.  of the 49th Symposium on Ring Theory and Representation Theory}, Symp. Ring Theory Represent. Theory Organ. Comm., Shimane, 2017, 131--134.

\bibitem{v-l}
M. R. Vaughan-Lee,
The restricted Burnside problem, \textit{Bull. Lond. Math. Soc.} \textbf{17} (1985), 113--133.

\bibitem{wal67}
G. E. Wall, On Hughes’ $H_p$ problem, \textit{Proc. Int. Conf. Theory Groups, Canberra},  1965, Gordon and Breach, New York, 1967, 357--362.

\bibitem{wal74}
G. E. Wall, On the Lie ring of a group of prime exponent, in: \textit{Proc. of the Second Int. Conf. on the
Theory of Groups, Canberra}, 1973, Lecture Notes in Math., vol. \textbf{372}, Springer, Berlin, 1974,  667--690.

\bibitem{wal75}
G. E. Wall, Secretive prime-power groups of large rank, \textit{Bull. Aust. Math. Soc.} \textbf{12} (1975), 363--369.

\bibitem{wal86}
G. E. Wall, On the multilinear identities which hold in the Lie ring of a group of prime-power exponent, \textit{J.~Algebra} \textbf{104}
(1986), 1--22.


\end{thebibliography}
\end{document}